\documentclass[12pt,a4paper]{article}
\usepackage{enumerate}
\usepackage{amsmath}
\usepackage{amsfonts}
\usepackage{amssymb}
\usepackage{amsthm}
\usepackage{color}
\usepackage{graphicx}
\usepackage{tabularx}
\usepackage{geometry}                          
\usepackage{stmaryrd}                         
\usepackage{hyperref}                           

\newtheorem{theorem}      {Theorem}[section]
\newtheorem*{theorem*}    {Theorem}
\newtheorem{proposition}  [theorem]{Proposition}

\newtheorem{lemma}        [theorem]{Lemma}

\newtheorem{remark}       [theorem]{Remark}

\newcommand{\R}{\mathbb{R}}     
\newcommand{\Z}{\mathbb{Z}}
\newcommand{\1}{\mathbf{1}}     
\newcommand{\E}{\mathbb{E}}     
\newcommand{\Var}{\text{Var}}   
\newcommand{\F}{\mathcal{F}}    

\renewcommand{\l}{\lambda}

\renewcommand{\phi}{\varphi}

\renewcommand{\H}{\bar{H}}
\newcommand{\V}{\text{Var}}
\newcommand{\qi}{\mbox{\large$\chi$}}
\newcommand{\f}[1]{\overrightarrow{#1}}

\title{Berry-Esseen bounds for the $\chi_2$-distance in the Central Limit Theorem: a Markovian approach}

\author{ Claire Delplancke\thanks{Center of Mathematical Modeling, Santiago, Chile, \texttt{cdelplancke@cmm.uchile.cl}.} \and Laurent Miclo\thanks{Institut de Math\'ematiques de Toulouse, UMR 5219, Universit\'e de Toulouse and CNRS, France, \texttt{laurent.miclo@math.univ-toulouse.fr}.} }

\begin{document}

\maketitle
\begin{abstract} 
This article presents a new proof of the rate of convergence to the normal distribution of sums of independent, identically distributed random variables in chi-square distance, which was also recently studied in \cite{BobkovRenyi}. Our method consists of taking advantage of the underlying time non-homogeneous Markovian structure and studying the spectral properties of the non-reversible transition operator, which allows to find the optimal rate in the convergence above under matching moments assumptions. Our main assumption is that the random variables involved in the sum are independent and have polynomial density; interestingly, our approach allows to relax the identical distribution hypothesis.


\paragraph{Mathematics Subject Classification (MSC2010):} 60F05, 60J05, 60J35, 47A75, 39B62.

\end{abstract}


\section{Introduction and main result}

The present article is devoted to the convergence in Central Limit Theorem with respect to $\chi_2$-distance, defined for two probability distributions $\theta,\,\mu$ as:
\begin{align*}
\qi_2(\theta,\mu)&:=\left(\int{\left(\frac{d\theta}{d\mu}-1\right)^2d\mu}\right)^\frac{1}{2}
\end{align*} 
if $\theta$ is absolutely continuous with respect to $\mu$ and $+\infty$ otherwise. In the following $\mu$ stands for the normal distribution. For a density function $f\in L^2(\mu)$ we use the shortened notation
\begin{align*}
\qi_2(f):=\qi_2(f\cdot\mu,\mu)
\end{align*}
to refer to the $\chi_2$-distance between the distribution with density $f$ with respect to $\mu$, denoted by $f\cdot \mu$, and $\mu$ itself. The $\chi_2$-distance bounds by above usual quantities like total variation distance and relative entropy:
\begin{align}
d_{TV}(f,\mu)&:=\int{|f-1|d\mu} \, \leq \,\qi_2(f),\quad \quad \text{Ent}(f||\mu):=\int{f\log f d\mu} \, \leq \, \qi_2^2(f),
\label{eq:DistComp}
\end{align}
the first relation being a consequence of Cauchy-Schwarz inequality and the second of the inequality $\log x \leq x-1$ for $x>0$.\smallskip 

Let $(X_i)_{i\geq 1}$ be real i.i.d.\@ random variables with density $\phi$ with respect to the normal distribution $\mu$, and consider the renormalized sum 
\begin{align}
Y_n&:=\frac{1}{\sqrt{n}}\sum_{i=1}^n{X_i},\quad \quad n\geq 1.
\label{eq:DefRenormalizedSum}
\end{align}
We call $f_n$ the density of $Y_n$ with respect to $\mu$.  The main result of this article is the following.
\begin{theorem}[Asymptotic bound on $\chi_2$]
If the moments of $X_1$ and the moments of $\mu$ match up to order $r$ for a given integer $r\geq 2$, and if the density $\phi$ of $X_1$ with respect to $\mu$ is polynomial and satisfies to Hypothesis \textbf{(H)} stated in Section \ref{sec:HermiteFourier}, then
\begin{align*}
\underset{n\rightarrow + \infty}{\lim \sup}\quad  n^{\frac{r-1}{2}}\,\qi_2(f_n) &< +\infty.  
\end{align*}
\label{thm:ConvergChi}
\end{theorem}

The rate of convergence, which improves by a factor $\sqrt{n}$ for each supplementary moment of $X_1$ that agrees with the corresponding moment of $\mu$, is optimal: indeed by inequality \eqref{eq:DistComp}, it implies a rate of at least $n^{(r-1)/2}$ in total variation distance, which is proved to be optimal in \cite{BallyVT}.\smallskip

While this article was being written, the authors took notice of the recent article \cite{BobkovRenyi}, which provides the same rate of convergence, as well as the constant in front of it, in the i.i.d.\@ case, and under optimal assumptions.\smallskip

Our method of proof starts from the simple observation that the sequence of renormalized sums $(Y_n)_{n \geq 1}$ is a non-homogeneous Markov chain, as it satisfies to the recursion equation
\begin{align}
Y_{n+1}&=\sqrt{1-\frac{1}{n+1}} Y_n+ \frac{1}{\sqrt{n+1}}X_{n+1},\quad n\geq 1.
\label{eq:RecursionSum}
\end{align} 
Although the result which we obtain in Theorem \ref{thm:ConvergChi} is weaker than that of \cite{BobkovRenyi} in the i.i.d.\@ case, this method presents two advantages:
\begin{enumerate}[-]
 \item The result can be extended to independent, non necessarily identically distributed random variables $(X_i)_{i\geq 1}$, as this does not affect the Markovian character of $(Y_n)_{n \geq 1}$, provided that the  $(X_i)_{i\geq 1}$ are distributed according to a finite set of distributions, each polynomial and satisfying to \textbf{(H)} (cf Remark \ref{rmk:NonIIDCase}). The idea to use a Markovian framework to deal with non identically distributed random variables has also recently been used in \cite{BallyNonIdenticallyDistributed}.

 \item It is original and rather straightforward. The action of the Markov transition semigroup amounts to a barycentric convolution, since
$$ \left(\sqrt{1-\frac{1}{n+1}} \right)^2 + \left( \frac{1}{\sqrt{n+1}} \right)^2=1.$$  From the point of view of functional analysis, our approach amounts to look at the bilinear barycentric convolution operator
 \begin{align*}
 (f,\phi) \mapsto af*b\,\phi,
 \end{align*}
where $f,\,\phi$ are two density functions in $L^2(\mu)$, $a,\,b\in \R$ such that $a^2+b^2=1$, and $af*b\,\phi$ denotes the density of the random variable $aU+bV$ whith $U\sim f \cdot \mu,\, V \sim \phi \cdot \mu$, as the linear operator (denoted $\mathcal{Q}_{a,\phi}^*$ in the following):
\begin{align*}
L^2(\mu) &\rightarrow L^2(\mu)\\
 f &\mapsto af*b\,\phi.
\end{align*}
The function $\phi$ is fixed and represents the density of the innovations $(X_i)_{i\geq 1}$. Thanks to spectral analysis in the Hermite-Fourier domain, we derive an estimate of the operator norm, which is directly related to the rate of convergence in Theorem \ref{thm:ConvergChi} and to Theorem \ref{thm:RefinedConvolChi} below. 
\end{enumerate}

\textit{More on the proof.--} At the heart of the study is a formula describing the evolution of the $\chi_2$-distance under the action of the barycentric convolution. Set $(\H_n)_{n\in {\Z^+}}$ the set of renormalized Hermite polynomials, forming an orthonormal basis of $L^2(\mu)$:
\begin{align*}
\H_n(x)&:=\frac{(-1)^n}{\sqrt{n!}} e^{\frac{x^2}{2}} D^n\left(e^{-\frac{x^2}{2}}\right),\quad \quad x\in \R,\quad n\in {\Z^+},
\end{align*}
where $D$ denotes the derivation operator acting on smooth functions from $\R$ to $\R$. We show the following result:

\begin{theorem}[Barycentric convolution and $\chi_2$]
Let $r$ be a natural integer and $f,\phi$ be two densities in $L^2(\mu)$ whose moments match the moments of $\mu$ up to order $r$, and assume moreover that the density $\phi$ is polynomial. In particular, $\phi$ admits a decomposition on the Hermite basis of the form:
\begin{align}
\phi=1+\sum_{k=r+1}^N {\phi_k \H_k}.
\label{eq:DecompDensity}
\end{align}
Set:
\begin{align*}
a_\phi:=\left(1+\frac{N}{r+1}\right)^{-\frac{1}{4}}\in (0,1],
\end{align*}
and for all $a\in (0,1)$,
\begin{align*}
d_\phi(a):=\sum_{k=r+1}^N{\frac{ |\phi_k|}{\sqrt{k!}} \left(-2(r+1+N)\left(1+\frac{N}{r+1}\right)\log a\right)^{k/2} }.
\end{align*}

If $\phi$ satisfies to Hypothesis \textbf{(H)} stated in Section \ref{sec:HermiteFourier}, then for all $a\in(a_\phi,1)$, the following inequality stands:
\begin{align}
\qi_2\left (af*\sqrt{1-a^2}\phi\right) &\leq a^{r+1}\left(1+d_\phi(a)\right)\,\qi_2(f)\,+\,(1-a^2)^{\frac{r+1}{2}}\,\qi_2(\phi).
\tag{$E_r$}
\label{eq:ConvolChi}
\end{align}
Moreover,
\begin{align*}
\lim_{a\rightarrow 1}|\log a|^{-\frac{r+1}{2}}d_\phi(a) &< +\infty.
\end{align*}
\label{thm:RefinedConvolChi}
\end{theorem}
 
Equation \eqref{eq:DecompDensity}, which involves Hermite coefficients $(\phi_k)_{r+1 \leq k \leq N}$, is explained in Section \ref{sec:HermiteFourier}. Hypothesis \textbf{(H)} gives conditions on these Hermite coefficients.

\smallskip

Theorem \ref{thm:RefinedConvolChi} does not preserve the symmetry of the equation $af*(1-a^2)^{1/2}\phi=(1-a^2)^{1/2}\phi*af$, in the first place because the assumption on $f$ is weaker than the assumption on $\phi$, and in the second place because the term $d_\phi(a)$ in the upper-bound is non-vanishing in general. However, in the regimen of interest $a \rightarrow 1$ (which corresponds to $(n/n+1)^{1/2} \rightarrow 1$), one has
$$a^{r+1}=1-\frac{r+1}{2}(1-a)+o(a^2),\quad \quad 1+d_\phi(a)=1+\mathcal{O}\left((1-a)^{\frac{r+1}{2}}\right),$$
hence in the Taylor expansion of the prefactor $a^{r+1}(1+d_\phi(a)) $ the contribution from $d_\phi(a)$ is negligible with respect to the contribution from $ a^{r+1}$ as soon as $r \geq 2$, i.e.\@ if the centering and normalizing condition $\E[X_1]=0,\, \E[X_1^2]=1$ is satisfied. Let us also mention that there exists an alternative inequality, holding without the polynomial assumption on $\phi$, stated in Section \ref{sec:ProofConvolChi}.

\smallskip

Bound \eqref{eq:ConvolChi} bears a similarity to Shannon-Stam inequality for (absolute) entropy (\cite{ShannonCommunication,StamInequalities}):
\begin{align}
\text{Ent}\left (a\,U+\sqrt{1-a^2}V\right) &\leq a \, \text{Ent}(U)+\sqrt{1-a^2}\, \text{Ent}(V),\quad \quad a\in [0,1],
\label{eq:ShannonStam}
\end{align}
although, by the observation above, the coefficients in front of $\text{Ent}$ are of different order with respect to the coefficients in front of $\chi$ for all natural integer $r$.

\smallskip

For $r=1$, bound \eqref{eq:ConvolChi} gives back Poincar\'e inequality for the Ornstein-Uhlenbeck semigroup $(P_t)_{t\geq 0}$, defined for $f\in L^2(\mu)$ as
\begin{align}
P_t[f](x)=\E\left[f\left(e^{-t}x+\sqrt{1-e^{-2t}}Z\right)\right],\quad Z \sim \mu;\quad \quad x\in \R,\quad t\geq 0.
\label{eq:OrnsteinUhlenbeck}
\end{align}
Indeed, in the case where $\phi$ is the density of the normal distribution $\mu$, that is to say $\phi=1$, we adopt the convention $r=+\infty$ and $N=0$ in equality \eqref{eq:DecompDensity}, and set $a_\phi=0$ and $d_\phi(a)=0$ for all $a \in [0,1]$. As we will see in Section \ref{sec:ProofConvolChi}, $(E_1)$ then corresponds (up to a positivity assumption which can actually be discarded in the proof) to Poincar\'e inequality for $(P_t)_{t\geq 0}$: for $f\in L^2(\mu)$,
\begin{align}
\V_\mu(P_t f) &\leq e^{-2t}\, \V_\mu(f),\quad \quad t\geq 0. \label{eq:PoincareOU} \\ \nonumber
\end{align}

\textit{Results of the literature.--} The first quantification result for the convergence of renormalized sums of i.i.d.\@ variables was obtained independently by Berry and Esseen through Kolmogorov distance (\cite{BerryAccuracy, EsseenDistribution}). Rate of convergence in total variation distance is first addressed in \cite{Shirazdinov} and the optimal rate under matching moment assumptions and regularity condition is proved in \cite{BallyVT} using Malliavin's calculus.\smallskip

The rate of convergence in Theorem \ref{thm:ConvergChi} relies crucially on the fact that the inequality \eqref{eq:ConvolChi} incorporates the matching moments assumption through exponent $r$ on the barycentric coefficients. By comparison, Shannon-Stam inequality \eqref{eq:ShannonStam} implies (jointly with a result of monotonicity of relative entropy and Fisher information under convolution) the convergence $\text{Ent}(f_n)\rightarrow 0$ without rate (\cite{BrownFisher,BarronEntropy}). The optimal rate is derived in \cite{ArtsteinEntropy} when the random variable $X_1$ satisfies to a Poincar\'e inequality. \cite{BobkovEntropyRate} provides an asymptotic expansion of entropy involving the moments of $X_1$. Similar developments occured for Fisher information (\cite{BrownFisher}, \cite{BobkovFisher}), and for R\'enyi distances (\cite{BobkovRenyi}), which include the $\chi_2$-distance. \smallskip

Other distances include Sobolev (\cite{GoudonToscani}) and Wasserstein (\cite{Ibragimov,Tanaka,RioBounds,RioConstants,RioStationarySequence}) distances. A typical assumption in Berry-Esseen theorems is the existence of moments up to a certain order for the random variable $X_1$. In the present framework, the fact that $\phi$ is in $L^2(\mu)$ implies that moments of all order exist. This is consistent with the fact that the $\chi_2$-distance bounds by above the usual quantities (as shown in \eqref{eq:DistComp}; a similar inequality holds for Wasserstein distance of order $1$). \smallskip

In another direction of research, Stein's method and Malliavin calculus revealed to be powerful tools to study, including in a quantitative way, the asymptotic normality of multidimensional random variables living in Gaussian chaoses. Let us cite, among an increasingly rich literature, the reference book \cite{NourdinSteinMalliavin}. It is interesting to remark the formal similarity between their objects and ours, although the results do not compare. Indeed, in the Stein-Malliavin framework, the typical random variable $X_1$ writes $X_1=\phi(Z)$, with $Z$ a Gaussian random variable living in $\R^n$ with law $\mu_n$ and $\phi\in L^2(\mu_n)$. To compare to our framework, let us take $n=1$ and $\phi$ a density in $L^2(\mu)$. The random variable $X_1=\phi(Z)$, where $Z\sim \mu$, bears no relation with the random variable $X_1$ which has \emph{density} $\phi$ with respect to $\mu$; hence, the two approaches are not reducible one to another. The authors would like to thank Anthony R\'eveillac for interesting discussions on this subject.\\

\textit{Structure of the article.--} The remaining of this paper is organized as follows. In Section \ref{sec:HermiteFourier}, Hermite-Fourier decomposition \eqref{eq:DecompDensity} is detailed, and Hypothesis \textbf{(H)} is stated and commented. Section \ref{sec:ConvolMarkov} is devoted to the explicit expression of the convolution operator and of its Hermite-Fourier decomposition. Theorem \ref{thm:RefinedConvolChi} is proved in Section \ref{sec:ProofConvolChi} by spectral analysis in the Hermite-Fourier domain. Finally Section \ref{sec:ProofRateChi} is devoted to the proof of Theorem \ref{thm:ConvergChi}.

\section{Hermite-Fourier decomposition of the density}
\label{sec:HermiteFourier}

First, let us introduce some notation. The symbol $\1$ stands for the function from $\R$ to $\R$ identically equal to $1$, $\Z^+$ for the set of natural integers and $\binom{n}{k}$ for the binomial coefficient associated to natural integers $k\leq n$.  For a fonction $f\in L^1(\mu)$, we denote indifferently
\begin{align*}
\mu(f)=\int{fd\mu}.
\end{align*}
The space $L^2(\mu)$ is a Hilbert space, with scalar product and associated norm defined as
\begin{align*}
\langle f,g \rangle_{L^2(\mu)} :=\int{fg\, d\mu},\quad \quad \|f\|_{L^2(\mu)}:=\sqrt{\mu(f^2)} ,\quad \quad f,g \in L^2(\mu).
\end{align*}
Set
\begin{align*}
\Var_\mu(f):=\int{(f-\mu(f))^2d\mu}.
\end{align*}

Hermite polynomials $(H_n)_{n\in {\Z^+}}$ are defined as: 
\begin{align*}
H_n(x)&:=(-1)^n e^{\frac{x^2}{2}} D^n\left(e^{-\frac{x^2}{2}}\right),\quad \quad x\in \R,\quad n\in {\Z^+},
\end{align*}
where we recall that $D$ stands for the derivation operator acting on smooth functions from $\R$ to $\R$. Hermite polynomials are also characterized by the following equation: for all smooth functions $f:\R\rightarrow\R$,
\begin{align}
\int{f H_n d\mu}=\int{D^n f d\mu}.
\label{eq:HermiteDerive}
\end{align}
They form an orthogonal basis of $L^2(\mu)$: for all $n,m \in {\Z^+}$,
\begin{align*}
\int{H_m H_n d\mu}= n! \,\delta_{n,m}.
\end{align*}
In the paper it is more convenient to work with renormalized Hermite polynomials $(\H_n)_{n\in {\Z^+}}$:
\begin{align*}
\H_n=\frac{1}{\sqrt{n!}}H_n,\quad \quad n \in {\Z^+},
\end{align*}
which form an orthonormal basis of $L^2(\mu)$. By convention set $\H_{-1}=0$. One has:
\begin{align}
D \H_n&= \sqrt{n} \H_{n-1},\quad \quad n \in {\Z^+}.
\label{eq:DeriveeHermite}
\end{align}
As $H_0=1$, for all natural integer $\H_n$ is of degree $n$. The basis $(\H_n)_{n \in {\Z^+}}$ is diagonal for the Ornstein-Uhlenbeck semigroup defined in \eqref{eq:OrnsteinUhlenbeck}:
\begin{align*}
P_t[\H_n]&=e^{-nt}\H_n,\quad \quad n \in {\Z^+},\quad t\geq 0.
\end{align*}

For all functions $g \in L^2(\gamma)$, call $(g_k)_{k\in {\Z^+}}$ its coefficients on the orthonormal Hermite basis,
\begin{align*}
g&=\sum_{k\in {\Z^+}}g_k \H_k,
\end{align*}
where the equality stands in $L^2(\mu)$, and denote indifferently $\F(g):=\f{g}=(g_k)_{k \in {\Z^+}}$ the sequence of its coefficients, which belongs to the Hilbert space $l^2$, defined as the set of real sequences $(u_k)_{k \in {\Z^+}}$ such that $\sum_{k\in {\Z^+}}u_k^2<+\infty$. The application
\begin{align}
L^2(\mu) \rightarrow l^2,\quad g \mapsto \F(g),
\label{eq:isometry}
\end{align}
is an isometry of Hilbert spaces.\smallskip

If $\phi \in L^2(\mu)$ is a density, then $\phi_0=1$. The matching moments assumption has a nice interpretation in terms of the coefficients: for all positive integer $r$,
\begin{align*}
\left( \forall k \in \left\{1,\dots r\right\},\quad \int{x^k\phi(x)d\mu(x)}=\int{x^kd\mu(x)}\right)\quad \Leftrightarrow\quad  \left( \forall k \in \left\{1,\dots r\right\},\quad \phi_k=0\right).
\end{align*}
Indeed, Hermite polynomial $\H_n$ being of degree $n$ for all natural integers, one has the equivalence
\begin{gather*}
\left( \forall k \in \left\{1,\dots r\right\},\quad \int{x^k\phi(x)d\mu(x)}=\int{x^kd\mu(x)}\right)\\
 \Leftrightarrow\quad  \left( \forall k \in \left\{1,\dots r\right\},\quad \int{\H_k(x)\phi(x)d\mu(x)}=\int{\H_k(x) d\mu(x)}\right),
\end{gather*}
and by orthogonality of $(\H_n)_{n\in {\Z^+}}$ it stands that for all $k \in {\Z^+}$,
\begin{align*}
\int{\H_k d\mu}&=\int{\H_k \H_0d\mu}=\delta_{0,k}.
\end{align*}
The assumption that $\phi$ is a polynomial density whose moments agree with moments of $\mu$ up to $r$ hence amounts to:
\begin{align*}
\exists N \in {\Z^+}, \quad N>r,\quad \phi=1+\sum_{k=r+1}^N{\phi_k \H_k},
\end{align*}
which corresponds to equality \eqref{eq:DecompDensity} above. From now on we set $K=r+1$. When $\phi$ is the density of $\mu$ itself, that is to say $\phi=\H_0=1$, we set $K=+\infty$ and $N=0$ by convention. \smallskip

Let us introduce the quantities
\begin{align*}
C_k&:=\left(1+\frac{N}{K}\right)^{k/2},\quad \quad \gamma_k=\frac{1}{\sqrt{k!}}C_k |\phi_k|,\quad\quad k \in {\Z^+}.
\end{align*}
We are now ready to state Hypothesis \textbf{(H)}, which is composed of two parts, \textbf{(H1)} and \textbf{(H2)} as follows. \smallskip

\textbf{(H1)} If $K\leq N-2$, for all $K\leq k\leq N-2$, 
$$(k+2)\gamma_{k+2}\leq \gamma_k.$$
For non-vanishing $\phi_k$ this relation is equivalent to 
$$\left|\frac{\phi_{k+2}}{\phi_k}\right|\leq \frac{1}{1+\frac{N}{K}}\left(1-\frac{1}{k+2}\right)^{1/2}$$
if $\phi_k\neq 0$ and is implied by the simplest assumption\smallskip

\textbf{(H1')} If $K\leq N-2$, for all $K\leq k\leq N$,
$$\left|{\phi_{k+1}}\right| \leq r \left|{\phi_k}\right|;\quad \quad  r:=\left(\left(1-\frac{1}{K+2}\right)^{1/2}\frac{1}{1+\frac{N}{K}}\right)^{1/2}.$$
Remark that assumption \textbf{(H1)} implies that $\gamma_k=0 \Rightarrow \gamma_{k+2}=0$. \smallskip

The second condition \textbf{(H2)} has two parts:\\
\textbf{(H2a)} If $K \leq N-1$, 
\begin{eqnarray*}
\left(\frac{2(K+1)\gamma_{K+1}}{N}\right)^N\left(\frac{K-1}{2\gamma_N}\right)^{K-1} \vee \, \left(\frac{2K\gamma_{K}}{N-1}\right)^{N-1}\left(\frac{K-2}{2\gamma_{N-1}}\right)^{K-2}
&\leq & \frac{1}{(N-K+1)^{N-K+1}}.
\end{eqnarray*}
\textbf{(H2b)} If $K=N$,
$$\gamma_N\leq \frac{1}{2(N-2)^{(N-2)/2}}.$$

\smallskip
Loosely speaking, assumption \textbf{(H1)}, which is present only if $K\leq N-2$, amounts to ask a geometric decrease for the coefficients $(\phi_k)$. If $K=N$ or $K=N-1$, assumption \textbf{(H2)} requires the leading coefficients $\phi_N$ and $\phi_{N-1}$ to be not too big, which is fair enough; it is more painful to write when $K<N-1$, but it can be interpreted as the requirement that the coefficients $(\phi_k)_{K \leq k \leq N}$ do not decay too fast.

\smallskip

We conclude this section by the following comment on the range of validity of Theorems \ref{thm:RefinedConvolChi} and \ref{thm:ConvergChi}.

\begin{remark}[On the polynomial assumption]
\label{rmq:PolynomialAssumption}

We conjecture that inequality \eqref{eq:ConvolChi} holds without Hypothesis \textbf{(H)}: indeed, the fact that $\phi$, as a density, is nonnegative already implies restrictions on the coefficients, which are in fact sufficient to prove \eqref{eq:ConvolChi} in the case of densities of the form $\phi=1+c\H_2$ and $\phi=1+c'\H_4$. \smallskip

Whether the polynomial assumption is necessary is less clear. For comparison, the hypothesis of Gaussian chaos of finite orders is needed in \cite{NourdinTV}.
\end{remark}

\section{Explicit expression of the convolution operator}
\label{sec:ConvolMarkov}
%

\subsection{Convolution as a Markovian transition}

Set $(X_i)_{i\geq 1}$ random variable of density $\phi \in L^2(\mu)$. Throughout the paper, the notation $f_n$ stands for the density of the renormalized sum 
\begin{align*}
Y_n&=\frac{1}{\sqrt{n}}\sum_{i=1}^n{X_i},\quad \quad n\geq 1,
\end{align*}
which satisfy to the recursion equation: 
\begin{align*}
Y_{n+1}&=\sqrt{1-\frac{1}{n+1}} Y_n+ \frac{1}{\sqrt{n+1}}X_{n+1},\quad n\geq 1.
\tag{\ref{eq:RecursionSum}}
\end{align*}
For a parameter $a\in [0,1]$, introduce the bilinear barycentric convolution operator $K_a$ defined as
\begin{align*}
\forall f,\phi \in L^2(\mu),\quad \quad K_a(f\,,\phi):=af*\sqrt{1-a^2}\,\phi.
\end{align*} 
Equation \eqref{eq:RecursionSum} in turn yields the corresponding recursion relation for the successive densities:
\begin{align}
f_{n+1}= K_{a_{n+1}}(f_n,\phi),\quad \quad a_{n+1}:=\sqrt{1-\frac{1}{n+1}},\quad n \geq 1.
\label{eq:RecursionDensity}
\end{align}
\smallskip

On the other hand, relation \eqref{eq:RecursionSum} translates into the the fact that $(Y_n)_{n\geq 1}$ is an inhomogeneous Markov chain. The associated semigroup $(\mathcal{Q}_{p,q})_{q\geq p \geq 1}$ is defined for continuous bounded functions as:
\begin{align*}
\mathcal{Q}_{p,q}[f](x)&:=\E[f(Y_q)|Y_p=x],\quad \quad x\in \R,\quad q\geq p \geq 1.
\end{align*}
By relation \eqref{eq:RecursionSum}, the explicit expression of $\mathcal{Q}_{n,n+1}$ is straightforward: for all $f\in L^2(\mu)$,
\begin{align*}
\mathcal{Q}_{n,n+1}[f]=Q_{a_{n+1},\phi}[f],
\end{align*}
where the operator $Q_{a,\phi}$ is defined for all $a \in [0,1]$ and density $\phi \in L^2(\mu)$ as:
\begin{align*}
Q_{a,\phi}[f](x)&:=\int{f(ax+\sqrt{1-a^2}y)\phi(y)dy},\quad \quad x \in \R.
\end{align*}
Denote $\mathcal{Q}_{n,n+1}^*$ (resp.  $  Q_{a,\phi}^*$) the adjoint of $\mathcal{Q}_{n,n+1}$ (resp. $Q_{a,\phi}$) in $L^2(\mu)$.
The discrete version of Kolmogorov backward relation reads:
\begin{align*}
f_{n+1}:=\mathcal{Q}_{n,n+1}^*[f_n],\quad \quad n\geq 1.
\end{align*}
Hence $\mathcal{Q}_{n,n+1}^*[f_n]=K_{a_{n+1}}(f_n,\phi)$. More generally, for $f,\phi$ densities in $L^2(\mu)$, the following identity holds:
\begin{align*}
K_a(f,\phi)&=Q_{a,\phi}^* [f],\quad \quad a\in[0,1].
\end{align*}
The strategy used in the paper relies on this interpretion of barycentric convolution as the action of a Markovian transition, as explained in Section \ref{sec:Strategy}. 

\begin{proposition}[Explicit expression of the operators]
For all $a\in [0,1]$, one has
\begin{align}
Q_{a,\phi}[f](x)&=\int{f(ax+\sqrt{1-a^2}y)\phi(y)d\mu(y)}, \label{eq:ExplicitQ}\\
Q_{a,\phi}^*[f](x)&=\int{f(ax-\sqrt{1-a^2}y)\phi(\sqrt{1-a^2}x+ay)d\mu(y)}=K_a(f,\phi)(x).
\label{eq:ExplicitQstar}
\end{align}
\label{prop:Explicit}
\end{proposition}

The formulas are to be understood in the following way: if $f$ is bounded and continuous, they stand for all $x\in \R$; if $f\in L^2(\mu)$, they stand in the almost everywhere sense. One sees that the operators $Q_{a,\phi}$ and $Q_{a,\phi}^*$ are actually defined for all $f,\phi \in L^2(\gamma)$ independently from them being densities; in what follows, $Q_{a,\phi}$ and $Q_{a,\phi}^*$ refers to this extended definition when required. Furthermore, $Q_{a,\phi}$ and $Q_{a,\phi}^*$ are bounded in $L^2(\mu)$ for all $\phi \in L^2(\mu)$: indeed by Cauchy-Schwarz inequality, for all $f\in L^2(\mu)$,
\begin{align*}
\int{\left(Q_{a,\phi}[f]\right)^2d\mu} \leq \|f\|^2_{L^2(\mu)} \|\phi\|^2_{L^2(\mu)},\quad \quad \int{\left(Q_{a,\phi}^*[f]\right)^2d\mu} \leq \|f\|^2_{L^2(\mu)} \|\phi\|^2_{L^2(\mu)}.
\end{align*}

\begin{proof}[Proof of Proposition \ref{prop:Explicit}]

Formula \eqref{eq:ExplicitQ} has already been given; let us prove formula \eqref{eq:ExplicitQstar}.
As $a \in [0,1]$, there exists $\theta \in \R$ such that $\cos \theta =a$ and $\sin \theta =\sqrt{1-a^2}$. Denote $R_\theta$ the rotation of $\R^2$ with parameter $\theta$ and $\Gamma$ the Gaussian distribution on $\R^2$ with identity as covariance matrix. Then, invariance of $\Gamma$ under the action of $R_\theta$ implies that for all $f,g \in L^2(\gamma)$,
\begin{align*}
\int{fQ_{a,\phi}[g]d\mu}&=\iint{f(x)\phi(y)g((\cos \theta) x+ (\sin \theta) y)d\mu(x)d\mu(y)}\\
&=\iint{f(\Re (X))\phi(\Im(X))g(\Re(R_\theta X))d\Gamma(X)}\\
&=\iint{f(\Re (R_{-\theta} X))\phi(\Im(R_{-\theta} X))g(\Re(X))d\Gamma(X)}=\int{gQ_{a,\phi}^*[f]d\mu}.
\end{align*}
\end{proof}

\smallskip

\begin{remark}[Link with Ornstein-Uhlenbeck semigroup]
Recall the definition of the Ornstein-Uhlenbeck semigroup $(P_t)_{t\geq 0}$ given in \eqref{eq:OrnsteinUhlenbeck}. Hence, for all $a\in [0,1]$ and $f,\phi \in L^2(\mu)$, one has the two useful equalities:
\begin{align*}
Q_{a,\1}^*[f]=K_a(f,\1)=P_{-\log a}[f],\quad \quad Q_{a,\phi}^*[\1]=K_a(\1,\phi)=P_{-\frac{1}{2}\log(1- a)^2}[\phi].
\end{align*}
\label{rmq:OrnsteinUhlenbeck}
\end{remark}

\smallskip


Let us introduce the multiplicative reversibilization of the Markov transition operator $Q_{a,\phi}$, defined as
\begin{align}
M_{a,\phi}:=Q_{a,\phi}Q_{a,\phi}^*.
\label{eq:RevDef}
\end{align}
The concept of reversibilization of a non-reversible Markov operator traces back to \cite{FillReverse} which deals with homogeneous, invariant Markov chains. The new operator $M_{a,\phi}$ is now symmetric in $L^2(\mu)$, though in the general case, it is not Markovian: as $Q_{a,\phi}$ is not invariant, then $Q_{a,\phi}^*[\1]\neq \1$ and the mass conservation property $M_{a,\phi}[\1]= \1$ does not hold. Nonetheless, we will see in Section \ref{sec:PoincareLike} that spectral analysis of $M_{a,\phi}$ gives quantitative information on the action of Markovian transition $Q_{a,\phi}$ and convolution operator $Q_{a,\phi}^*$.

\subsection{Hermite-Fourier decomposition of the convolution operator}

Let us now determine how the operators $Q_{a,\phi}$, $Q_{a,\phi}^*$ and $M_{a,\phi}$ act with respect to the Hermite-Fourier decomposition defined in \eqref{eq:isometry}. For a bounded operator $Q$ of $L^2(\mu)$, we call $\f{Q}$ the infinite matrix defined as:
\begin{align*}
\f{Q}:=(\f{Q}(m,n))_{n,m\in {\Z^+}};\quad \quad (\f{Q})(m,n):=\int{Q(\H_n)\H_m d\mu},\quad \quad n,m\in {\Z^+}.
\end{align*}
The matrix $\f{Q}$ is the unique bounded operator of $l^2$ such that
\begin{align*}
\F(Q[f])= \f{Q} \f{f},\quad \quad f\in L^2(\mu).
\end{align*}
Set $\|R\|_\text{op}$ the operator norm of a bounded operator $R$ on a Hilbert space $\mathcal{H}$, defined with evident notation as
\begin{align*}
\|R\|_\text{op}:= \sup_{h \in \mathcal{H} \setminus \left\{0\right\}}\frac{\|Rh\|_\mathcal{H}}{\|h\|_\mathcal{H}}.
\end{align*}
The following property reveals useful: for all bounded operator $Q$ on $L^2(\mu)$, it stands that:
\begin{align}
\|Q\|_\text{op}=\|\f{Q}\|_\text{op}.
\label{eq:EgaliteNormeOP}
\end{align}

In the following proposition, we give the matrices $\f{Q}_{a,\phi}$, $\f{Q^*}_{a,\phi}$ and $\f{M}_{a,\phi}$ associated to the operators $Q_{a,\phi}$, $Q_{a,\phi}^*$ and $M_{a,\phi}$.

\begin{proposition}[Matrix form of operators]
For all $\phi\in L^2(\mu)$ and $a\in [0,1]$, one has:
\begin{align}
\forall m,n \in {\Z^+},\quad \f{Q}_{a,\phi}(m,n)&=\begin{cases}\binom{n}{m}^{\frac{1}{2}} a^{m}\left(1-a^2\right)^{\frac{n-m}{2}}  \phi_{n-m} ,& m \leq n\\ 0,& m>n,\end{cases} \label{eq:CoeffQ}
\end{align}
Denoting ${}^T\!N$ the transpose matrix of $N$, it stands that:
\begin{align*}
\f{Q^*}_{a,\phi} = {}^T\f{Q}_{a,\phi}.
\end{align*}
Finally, the matrix $\f{M}_{a,\phi}$ is symmetric and
\begin{align}
\forall l,i\in {\Z^+},\, i \leq l,\quad\f{M}_{a,\phi}(l,l-i)&=a^{2l-i}\sum_{k\geq 0}{\binom{k+l}{k}^{1/2}\binom{k+l}{k+i}^{1/2}(1-a^2)^{\frac{2k+i}{2}}\phi_{k+i}\phi_k}.\label{eq:CoeffM}
\end{align}
\end{proposition}

\begin{proof}
To begin with, one needs to compute the Hermite-Fourier decomposition of $Q_{a,\H_m}[\H_n]$, for $a\in [0,1]$ and $n,m \in {\Z^+}$. Applying the properties of Hermite polynomials recalled in Section \ref{sec:HermiteFourier} yields:
\begin{align*}
Q_{a,\H_m}[\H_n]&=\int{\H_n(ax+\sqrt{1-a^2}y)\H_m(y)d\mu(y)}=\frac{1}{\sqrt{n!}\sqrt{m!}}\int{D^m(H_n(ax+\sqrt{1-a^2}y))d\mu(y)}.
\end{align*}
By the degree property, the integral vanishes for $n<m$. For $n\geq m$,
\begin{align*}
Q_{a,\H_m}[\H_n]&=\left(1-a^2\right)^{\frac{m}{2}}\frac{n \cdots (n-m+1)}{\sqrt{n!}\sqrt{m!}}\int{H_{n-m}(ax+\sqrt{1-a^2}y))d\mu(y)}\\
&=\left(1-a^2\right)^{\frac{m}{2}}\frac{n \cdots (n-m+1)}{\sqrt{n!}\sqrt{m!}}P_{-\log a}[H_{n-m}]\\
&=a^{n-m}\left(1-a^2\right)^{\frac{m}{2}}\frac{n \cdots (n-m+1)}{\sqrt{n!}\sqrt{m!}}H_{n-m}\\
&=a^{n-m}\left(1-a^2\right)^{\frac{m}{2}} \binom{n}{m}^{\frac{1}{2}}\H_{n-m}.
\end{align*}
By bilinearity, write 
\begin{align*}
Q_{a,\phi}(\H_n)&=\sum_{m\in {\Z^+}}{\phi_m Q_{a,\H_m}[\H_n]}=\sum_{m=0}^n{\phi_m a^{n-m}\left(1-a^2\right)^{\frac{m}{2}} \binom{n}{m}^{\frac{1}{2}}\H_{n-m} }\\
&=\sum_{m=0}^n{\phi_{n-m} a^{m}\left(1-a^2\right)^{\frac{n-m}{2}} \binom{n}{m}^{\frac{1}{2}}\H_{m} } =\sum_{m=0}^n{\f{Q}_{a,\phi}(m,n)\H_{m} },
\end{align*}
which proves \eqref{eq:CoeffQ}. Furthermore, definition \eqref{eq:RevDef} implies that $\f{M}_{a,\phi}=\f{Q}_{a,\phi} {}^T\! \f{Q}_{a,\phi}$, and formula \eqref{eq:CoeffM} follows by simple computation.
\end{proof}

\begin{remark}[On Hermite decomposition of the convolution operator]\
\begin{itemize}
\item It is to be noted that convolution with barycentric coefficients only admits a nice decomposition; contrarily to what happens with Fourier transform associated to Lebesgue measure, the usual convolution has no explicit Hermite-Fourier representation.

\item Thanks to formula \eqref{eq:CoeffQ} above, one finds that for all $\phi \in L^2(\mu)$, $m\in {\Z^+}$ and $a\in [0,1]$,
\begin{align*}
Q_{a,\phi}^*(\H_m)&=a^m \sum_{n\geq 0}{\binom{m+n}{n}^{1/2}(1-a^2)^{\frac{n}{2}}\phi_n \H_{m+n}},
\end{align*}
which allows to better understand the behaviour of the barycentric convolution: each nonvanishing coefficient on $\H_m$ and $\H_n$ in the respective decompositions of $f$ and $\phi$ contribute to a coefficient on $\H_{m+n}$ in the decomposition of $K_a(f,\phi)$.\smallskip

\item If $\phi$ is polynomial, then $\f{M}_{a,\phi}$ is a band matrix.
\end{itemize}
\label{rmq:HermiteConvol}
\end{remark}

\smallskip

We already noticed that $Q_{a,\phi}$ and $Q_{a,\phi}^*$, and by composition $M_{a,\phi}=Q_{a,\phi}Q_{a,\phi}^*$, are bounded operators. In fact, they are Hilbert-Schmidt operators. By definition, a bounded operator $R$ on the Hilbert space $\mathcal{H}$ is Hilbert-Schmidt, if, $(e_n)_{n\in {\Z^+}}$ standing for an orthonormal basis of $\mathcal{H}$, one has:
\begin{align*}
\sum_{n\in {\Z^+}}\|R(e_n)\|_\mathcal{H}^2 < +\infty.
\end{align*}

\begin{proposition}[Hilbert-Schmidt operators]
For all $\phi\in L^2(\mu)$ and $a\in [0,1)$, the operators $Q_{a,\phi},Q_{a,\phi}^*,M_{a,\phi}$ are Hilbert-Schmidt, hence compact.
\label{prop:HilbertSchmidt}
\end{proposition}

\begin{proof}
Consider first $Q_{a,\phi}^*$.
\begin{align*}
\sum_{n \in {\Z^+}}{\|Q_{a,\phi}^*(\H_n)\|_{L^2(\mu)}^2}&=\sum_{n \in {\Z^+}}\langle \H_n,M_{a,\phi}\H_n \rangle_{L^2(\mu)}=\sum_{n\in {\Z^+}}\f{M}_{a,\phi}(n,n)\\
&=\underset{n,k \in {\Z^+}}{\sum}{\binom{k+n}{k}  (1-a^2)^k a^{2n}\phi_k^2}.
\end{align*}
Now, by the equality 
\begin{align*}
\sum_{n \in {\Z^+}}\binom{k+n}{k}u^n &=\frac{1}{(1-u)^{k+1}},\quad \quad k\in \Z^+,\quad u \in [0,1),
\end{align*}
we find that
\begin{align*}
\sum_{n \in {\Z^+}}{\|Q_{a,\phi}^*(\H_n)\|_{L^2(\mu)}^2}&=\sum_{k \in {\Z^+}}{ \frac{(1-a^2)^k}{(1-a^2)^{k+1}} \phi_k^2}=\frac{1}{1-a^2}\|\phi\|_{L^2(\mu)}^2 < +\infty.
\end{align*}
This implies that $Q_{a,\phi}$ is Hilbert-Schmidt and in turn so is $M_{a,\phi}$ by composition.
\end{proof}

\section[Proof of the convolution theorem]{Proof of Theorem \ref{thm:RefinedConvolChi}}
\label{sec:ProofConvolChi}

\subsection{Strategy for non-homogeneous Markov chains}
\label{sec:Strategy}

Let us now explain the strategy to exploit the Markovian framework. In Remark \ref{rmq:OrnsteinUhlenbeck}, we noticed that if $\phi=\1$, \emph{i.e.} the $X_i$'s are normal, then $Q_{a,\1}^*=P_{-\log a}$. In this case, the renormalized sums $(Y_n)_{\geq 1}$ are also normal, which corresponds to the fact that the Ornstein-Uhlenbeck semigroup $(P_t)_{t\geq 0}$ is invariant, and in fact reversible, with respect to $\mu$. The semigroup $(P_t)_{t\geq 0}$ also enjoys a Poincar\'e inequality recalled in equation \eqref{eq:PoincareOU}. If $f$ is a density then $\Var_\mu(f)=\chi^2(f)$, hence Poincar\'e inequality for the Ornstein-Uhlenbeck semigroup reads:
\begin{align*}
\qi_2(P_t [f]) &\leq e^{-t}\, \qi_2(f),\quad \quad t\geq 0.
\end{align*}
Furthermore, $f_{n+1}=P_{-\log a_{n+1}} [f_n]$ by \eqref{eq:RecursionDensity}, hence:
\begin{align*}
\qi_2(f_{n+1}) &= \qi_2\left( P_{-\log a_{n+1}}[f_n]   \right) \,\leq\, a_n \,\qi_2(f_n),\quad \quad n \geq 1,
\end{align*}
and by straighforward calculation one gets the decrease of $\qi_2(f_n)$.
\smallskip

The idea underlying our method consists in mimicking the reasoning above for the true operator $Q_{a,\phi}^*$ acting on densities, which is neither reversible nor satisfies to the mass conservation property in the general case, as was explained above. For $a\in[0,1]$ and $f$ a density in $L^2(\mu)$, let us write by triangular inequality:
\begin{align*}
\qi_2(Q_{a,\phi}^*[f])& = \left\|  Q_{a,\phi}^*[f] - 1 \right\|_{L^2(\mu)}\\
&\leq \left\|  Q_{a,\phi}^*[f-1]  \right\|_{L^2(\mu)} +\left\|  Q_{a,\phi}^*[\1] - 1 \right\|_{L^2(\mu)}.
\end{align*}
The term
\begin{align*}
\left\|  Q_{a,\phi}^*[\1] - 1 \right\|_{L^2(\mu)},
\end{align*}
can be thought of as a measure of the divergence from invariance of the transition operator, an idea tracing back to \cite{MicloAlgorithm}. \smallskip

Second, the centered term $\left\|  Q_{a,\phi}^*[f-1]  \right\|_{L^2(\mu)}$ rewrites:
\begin{align*}
\left\|  Q_{a,\phi}^*[f-1]  \right\|_{L^2(\mu)}^2&=\int{\left(  Q_{a,\phi}^*[f-1] \right)^2d\mu}=\int{(f-1) Q_{a,\phi} Q_{a,\phi}^*[f-1] d\mu}\\
&=\int{(f-1) M_{a,\phi}[f-1] d\mu},
\end{align*}
making appear the multiplicative reversibilization $M_{a,\phi}$ of $Q_{a,\phi}^*$ introduced above.\smallskip

In terms of convolution, this amounts to consider separately $K_a(f-1,\phi)$ and $K_a(\1,\phi)$. \smallskip

Following this roadmap, Proposition \ref{prop:DefaultInv} below deals with the default of invariance $\|Q_{a,\phi}^*[\1]-1\|_{L^2(\mu)}$ and Proposition \ref{prop:PoincareReverse} with the centered quantity $\|Q_{a,\phi}^*[f-1]\|_{L^2(\mu)}$. Theorem \ref{thm:RefinedConvolChi} is then proved in Section \ref{sec:ProofTheoremConvolChi}. For the sake of completeness, we conclude the part by stating an alternative bound to \eqref{eq:ConvolChi} in Section \ref{sec:AlterBound}.

\subsection{Improved Poincar\'e inequality for Ornstein-Uhlenbeck}
\label{sec:BoundDefaults}

The following result is an improvement of the usual Poincar\'e inequality for the Ornstein-Uhlenbeck semigroup \eqref{eq:PoincareOU} when more information is avalaible on the function $f\in L^2(\mu)$ at play.

\begin{proposition}[Improved Poincar\'e]
Let $r\in {\Z^+}$ and $f$ be a function in $L^2(\mu)$ with Hermite decomposition of the form
\begin{align*}
f=f_0+\sum_{n \geq r+1}{f_k \H_k},
\end{align*} 
Then for all $t \geq 0$,
\begin{align*}
\Var_\mu(P_t f) &\leq e^{-2(r+1)t} \Var_\mu(f).
\end{align*}
In particular, if $\phi$ is the density of a variable agreeing with the Gaussian moments up to $r$, then for all $a\in [0,1]$,
\begin{align*}
\|Q_{a,\phi}^*[\1]-1\|_{L^2(\mu)} &\leq (1-a^2)^\frac{r+1}{2} \qi_2(\phi).
\end{align*}
\label{prop:DefaultInv}
\end{proposition}

\begin{proof}
Thanks to the properties of Hermite polynomials,
\begin{align*}
P_t f&= \sum_{k=0}^\infty{f_k P_{t}[\H_k]}=h_0 +\sum_{k=r+1}^\infty{f_k e^{-kt}\H_k};\\
\Var_\mu(P_t f)&=\sum_{k=r+1}^\infty{f_k^2 e^{-2kt}} \leq e^{-2(r+1)t} \sum_{k=r+1}^\infty{f_k^2} = e^{-2(r+1)t}\Var_\mu(f)  .
\end{align*}
The second inequality of Proposition \ref{prop:DefaultInv} follows from the first one by Remark \ref{rmq:OrnsteinUhlenbeck}.
\end{proof}

\subsection{Poincar\'e-like inequality for the convolution operator}
\label{sec:PoincareLike}
For $a\in[0,1]$ and a centered $g\in L^2(\mu)$ (that is $\mu(g)=0$), let us consider the quantity $\|Q_{a,\phi}^*[g]\|_{L^2(\mu)}$. By analogy with the Poincar\'e inquality for the Ornstein-Uhlenbeck semigroup, which is reversible with respect to $\mu$, we call the following result a Poincar\'e-like inequality holding for the operator $Q_{a,\phi}^*$, which in general is non-reversible. 

\begin{proposition}[Poincar\'e-like inequality]
Assume that $\phi$ is a  polynomial density in $L^2(\mu)$ whose moments match the moments of $\mu$ up to order $r \in {\Z^+}$, and which satisfies to Hypothesis \textbf{(H)} stated in Section \ref{sec:HermiteFourier}. Set $a_\phi\in [0,1)$ and $d_\phi:(0,1) \rightarrow \R$ as in Theorem \ref{thm:RefinedConvolChi}. 
Then, for all function $g\in L^2(\mu)$ which writes as: 
\begin{align*}
g=\sum_{k=r+1}^\infty g_k \H_k,
\end{align*}
and for all $a\in (a_\phi,1)$, it stands that:
\begin{align*}
\int{(Q_{a,\phi}^*[g])^2 d\mu} & \leq a^{r+1}\left(1+d_\phi(a)\right)^2 \int{g^2 d\mu} .
\end{align*}
\label{prop:PoincareReverse}
\end{proposition}

The proof is cut out in a number of steps. We begin by the following lemma:

\begin{lemma}[Gershgorin's theorem]
Let $\phi$ and $g$ be as in Proposition \ref{prop:PoincareReverse}, and set $K=r+1$. Then, for all $a\in (0,1)$,
\begin{align}
\int{(Q_{a,\phi}^*[g])^2 d\mu} & \leq \sup_{l \geq K}\,\Sigma_{a,\phi}(l)\,\int{g^2 d\mu},
\label{eq:AppliGershgorin}
\end{align}
where
\begin{align*}
\Sigma_{a,\phi}(l):={\sum_{j=K}^{+\infty} \left|\f{M}_{a,\phi}(l,j)\right|}, \quad \quad l \geq K.
\end{align*}
\label{lemma:Gershgorin}
\end{lemma}

\begin{proof}
Let $r\in {\Z^+}$, $a\in (0,1)$ and $\phi$ as in the statement of the proposition, and let $K=r+1$. First, notice that $V_K$, the set of functions $g\in L^2(\mu)$ with Hermite decomposition
\begin{align*}
g=\sum_{k=K}^\infty g_k \H_k,
\end{align*}
is stable under action of $Q_{a,\phi}^*$ by Remark \ref{rmq:HermiteConvol}. The space $V_K$ equipped with the $L^2(\mu)$ structure is again a Hilbert space. Let us call ${Q_{a,\phi}^*}_{|_{V_K}}$ the restriction of $Q_{a,\phi}^*$ to $V_K$. It is again bounded, with operator norm
\begin{align*}
\left\|{Q_{a,\phi}^*}_{|_{V_K}}\right\|_{ \text{op}}:=\sup_{g \in V_K\setminus\left\{0\right\}}\frac{\|Q_{a,\phi}^*[g]\|_{L^2(\mu)}}{\|g\|_{L^2(\mu)}}.
\end{align*} 
Hence, the desired majoration \eqref{eq:AppliGershgorin} is equivalent to the following bound on the operator norm:
\begin{align}
\left\|{Q_{a,\phi}^*}_{|_{V_K}}\right\|_{ \text{op}}^2 & \,\leq \, \sup_{l \geq K}\,\Sigma_{a,\phi}(l).
\label{eq:EtapeInter}
\end{align}
The isometry between $L^2(\mu)$ and $l^2$ restricts to an isometry between $V_K$ and $l^2_K$, defined as the space of real sequences $(u_n)_{n\geq K}$ with $\sum_{n\geq K}u_n^2<+\infty$. By this isometry, if $N_K=(N_K(i,j))_{i,j \geq K}$ stands for the infinite matrix associated to ${Q_{a,\phi}^*}_{|_{V_K}}$, then:
\begin{align*}
\left\|{Q_{a,\phi}^*}_{|_{V_K}}\right\|_{ \text{op}}=\left\|N_K \right\|_{ \text{op}}.
\end{align*}
Furthermore, by the properties of block matrix multiplication, one sees that the matrix ${}^T\!N_K N_K$ is nothing else but the matrix $\f{M}_{a,\phi}$ defined in \eqref{eq:CoeffM} (Section \ref{sec:ConvolMarkov}) restricted to $l^2_K$, that is:
\begin{align*}
{}^T\!N_K N_K=\left(\f{M}_{a,\phi}(i,j)\right)_{i,j \geq K}.
\end{align*}

\smallskip

For a complex Banach space $E$, set $\mathcal{G}(E)$ the set of inversible operators on $E$. The spectral radius of a bounded operator $M$ in $E$ is then defined as 
\begin{align*}
\rho(M):=\max\left\{ |\lambda|,\, \lambda \text{Id}-M \in \mathcal{G}(E)\right\}.
\end{align*}
Moreover, if $E$ is Hilbert and if $T$ is a bounded operator of $E$ with adjoint $T^*$, then
\begin{align*}
\|T\|_{ \text{op}}=\sqrt{\rho(T^*T)}.
\end{align*}
The operator $N_K$ being a bounded operator of $l_K^2$ (by restriction of a bounded operator), the preceding equation applies:
\begin{align*}
\left\|N_K \right\|_{ \text{op}}=\sqrt{\rho({}^T\!N_K N_K)}.
\end{align*}

\smallskip

Let us recall a theorem of Gershgorin (\cite{Gershgorin}) related to finite complex auto-adjoint matrices $A$, where $A=(A_{i,j})_{1\leq i,j \leq n}$ for a positive integer $n$. Denoting $\mathcal{G}_n(\mathbb{C})$ the set of invertible matrices of size $n$ and $B(x,r)$ the complex ball of center $x\in \mathbb{C}$ and $r>0$, one has:
\begin{align*}
\left\{\l\in \mathbb{C},\, \l\text{Id}-A \in \mathcal{G}_n(\mathbb{C})\right\} \subset \bigcup_{l=1}^n{B\left(A(l,l),\left|\sum_{1\leq j \leq n,\,j\neq l}A(l,j)\right|\right)}.
\end{align*}
As a consequence,
\begin{align*}
\rho(A) &\leq \sup_{1\leq l\leq n}\left|A(l,l)\right| + \left|\sum_{1\leq j \leq n,\,j\neq l}A(l,j)\right| \leq \sup_{1\leq l\leq n}\sum_{j=1}^n|A(l,j)|.
\end{align*}
Gershgorin's theorem is stated for finite matrices, but the proof extends without difficulty to eigenvalues of operators on $l_K^2$. The operator ${}^T\!N_K N_K$ being autoadjoint and compact by Proposition \ref{prop:HilbertSchmidt}, its spectrum is included in the set of eigenvalues united with the singleton $\left\{0\right\}$, hence the formula above applies and yields majoration \eqref{eq:EtapeInter}, which proves the lemma.

\end{proof}

In order to derive an upper-bound of $\sup_{l \geq K}\,\Sigma_{a,\phi}(l)$ from the explicit expression of $\f{M}_{a,\phi}$ stated in \eqref{eq:CoeffM}, we need two technical lemmas.

\smallskip

\begin{lemma}[First technical lemma]
\label{lemma:DoubleProd}
Let $N \geq K$ be positive integers, and call as in the preceding sections
\begin{align*}
a_\phi:=\left(1+\frac{N}{K}\right)^{-\frac{1}{4}},\quad \quad C_k:=\left(1+\frac{N}{K}\right)^{k/2},\quad \quad k \in {\Z^+}.
\end{align*}
Let $i,k$ be natural integers such that
\begin{align*}
0\leq k \leq N-1,\quad \quad K\leq i+k \leq N,\quad \quad 1\leq i \leq N.
\end{align*}
Then, for all $a \in (a_\phi,1)$ and for all $l \geq K$,
\begin{gather*}
a^{-i}\binom{k+l}{k}^{1/2}\binom{k+l}{k+i}^{1/2}\mathbf{1}_{l \geq i+K}+a^i\binom{k+l+i}{k}^{1/2}\binom{k+l+i}{k+i}^{1/2}\\
\leq 2 C_k C_{k+i} \binom{k+l}{k}^{1/2}\binom{k+l+i}{k+i}^{1/2}.
\end{gather*}
\end{lemma}

\begin{proof}
Let $N \geq K$ be positive integers and $i,k,l$ be natural integers such that
\begin{align*}
0\leq k \leq N-1,\quad \quad K\leq i+k \leq N,\quad \quad 1\leq i \leq N, \quad \quad l \geq K.
\end{align*}
For two positive integers $m \geq n$, the notation $[m]_n$ stands for $[m]_n:=m \cdots (m-n+1)$. One has:
\begin{align*}
\binom{k+i+l}{k} \binom{k+l}{k}^{-1}=\frac{[k+i+l]_k}{[k+l]_k}.
\end{align*}
If $i\geq k$, then
\begin{align*}
\binom{k+i+l}{k} \binom{k+l}{k}^{-1}\leq \left(\frac{l+N}{l+1}\right)^k\leq  \left(\frac{l+N}{l+1}\right)^i.
\end{align*}
If $i<k$, then
\begin{align*}
\binom{k+i+l}{k} \binom{k+l}{k}^{-1}=\frac{[k+i+l]_i\,[k+l]_{k-i}}{[k+l]_{k-i}\,[l+i]_i}=\frac{[k+i+l]_i}{[i+l]_i}\leq \left(\frac{l+N}{l+1}\right)^i.
\end{align*}
In both cases,
\begin{align*}
\binom{k+i+l}{k} \binom{k+l}{k}^{-1}\leq \left(\frac{l+N}{l+1}\right)^i.
\end{align*}
Furthermore, in the case where $l \geq i$,
\begin{align*}
\binom{k+l}{k+i}\binom{k+i+l}{k+i}^{-1}=\frac{[k+l]_k\,[l]_i}{[k+i+l]_i\,[k+l]_k}=\frac{[l]_i}{[k+i+l]_i}\leq \left(\frac{l}{l+1}\right)^i.
\end{align*}
Hence, for all $a\in (0,1)$,
\begin{align}
&a^{-i}\binom{k+l}{k}^{1/2}\binom{k+l}{k+i}^{1/2}\mathbf{1}_{l \geq i+K}+a^i\binom{k+l+i}{k}^{1/2}\binom{k+l+i}{k+i}^{1/2}\nonumber \\
&\leq \left(\frac{l}{l+1}\right)^{i/2}  \binom{k+l}{k}^{1/2}\binom{k+l+i}{k+i}^{1/2} \left(a^{-i}\mathbf{1}_{l \geq i+K} + a^i\left(1+\frac{N}{l}\right)^{i/2} \right)\nonumber \\
&\leq  \binom{k+l}{k}^{1/2}\binom{k+l+i}{k+i}^{1/2} f(a),
\label{eq:EtapeInter2}
\end{align}
where we defined 
\begin{align*}
f(a):=a^{-i}+\left(1+\frac{N}{K}\right)^{i/2}a^i,\quad \quad a\in (0,1).
\end{align*}
As one checks easily, the inequality \eqref{eq:EtapeInter2} still holds true if $l <i$, and $f'(a) \geq 0$ if and only if $a \geq a_\phi=(1+N/K)^{-1/4}$. This yields for all $a \in (a_\phi,1)$,
\begin{align*}
f(a) &\leq f(1)= 1+\left(1+\frac{N}{K}\right)^{i/2} \leq  2\left(1+\frac{N}{K}\right)^{i/2} \leq  2 C_kC_{k+i},
\end{align*}
where $C_k$ has been defined as $C_k=\left(1+N/K\right)^{k/2}$ for all $k\in {\Z^+}$.
\end{proof}

\smallskip

\begin{lemma}[Second technical lemma]
\label{lemma:Polynom}
Let $m>q$ be positive integers, and consider the polynomial $P=-\alpha X^m+\beta X^q-1$ with $\alpha,\beta >0$. Then $P\leq 0$ on $\mathbb{R}^+$ if and only if
$$\left(\frac{\beta}{m}\right)^m\left(\frac{q}{\alpha}\right)^q \leq \frac{1}{(m-q)^{m-q}}.$$
\end{lemma}

\begin{proof}[Proof of Lemma \ref{lemma:Polynom}]
Let $P(x)=-\alpha x^m+ \beta  x^q-1$ be as in the wording of the lemma. Then, for all $x\in \R$, 
\begin{eqnarray*}
P'(x)&=&-m\alpha x^{m-1}+q\beta x^{q-1}=x^{q-1}(-m\alpha x^{m-q}+q\beta),
\end{eqnarray*}
thus $P$ attains its maximum on $[0,+\infty)$ at the point $x_0=\left((\beta q)/(\alpha m)\right)^{1/(m-q)}$. Moreover,
\begin{eqnarray*}
P(x_0)&=&x_0^q(-\alpha x_0^{m-q}+\beta)-1= \left(\frac{\beta q}{\alpha m}\right)^{\frac{q}{m-q}}  \left(-\frac{\beta q}{m}+\beta\right)-1\\
&=&\frac{\beta}{m}\left(\frac{\beta q}{\alpha m}\right)^{\frac{q}{m-q}} (m-q)-1=\left(\frac{\beta}{m}\right)^{\frac{m}{m-q}}\left(\frac{q}{\alpha}\right)^{\frac{q}{m-q}}(m-q)-1,
\end{eqnarray*}
so that $P(x_0) \leq 0$ if and only if 
$$\left(\frac{\beta}{m}\right)^m\left(\frac{q}{\alpha}\right)^q \leq \frac{1}{(m-q)^{m-q}}.$$
\end{proof}

\smallskip

We are now ready to show Proposition \ref{prop:PoincareReverse}.
\begin{proof}[Proof of Proposition \ref{prop:PoincareReverse}]
Set $K=r+1$, let $a \in (0,1)$ and $l$ a positive integer such that $l \geq K$. Then,
\begin{align*}
\Sigma_{a,\phi}(l)&=\sum_{j=K}^{+\infty} \left|\f{M}_{a,\phi}(l,j)\right|\\
& =\left|\f{M}_{a,\phi}(l,l)\right|+ \sum_{i=1}^N \left(\left|\f{M}_{a,\phi}(l,l-i)\right|\1_{l-i\geq K}+ \left|\f{M}_{a,\phi}(l,l+i)\right|\right)\\
&=a^{2l}\sum_{k\geq 0}{\binom{k+l}{k}(1-a^2)^{k}\phi_k^2}\\
&\quad+ \sum_{i=1}^N a^{2l-i}\sum_{k\geq 0}{\binom{k+l}{k}^{1/2}\binom{k+l}{k+i}^{1/2}(1-a^2)^{\frac{2k+i}{2}}\phi_{k+i}\phi_k}\1_{l-i\geq K}\\
&\quad +\sum_{i=1}^N a^{2l+i}\sum_{k\geq 0}{\binom{k+l+i}{k}^{1/2}\binom{k+l+i}{k+i}^{1/2}(1-a^2)^{\frac{2k+i}{2}}\phi_{k+i}\phi_k}\\
&=a^{2l}\left(1+\sum_{k=K}^N{\binom{k+l}{k}(1-a^2)^k}\phi_k^2 + \underset{0\leq k< k+i \leq N}{\sum}\left\{ (1-a^2)^{\frac{2k+i}{2}}|\phi_k| |\phi_{k+i}|\mathcal{C}_{a,\phi}(l,i,k)\right\}\right),
\end{align*}
where
\begin{align*}
\mathcal{C}_{a,\phi}(l,i,k)&:=a^{-i}\binom{k+l}{k}^{1/2}\binom{k+l}{k+i}^{1/2}\mathbf{1}_{l \geq i+K}+a^{i}\binom{k+l+i}{k}^{1/2}\binom{k+l+i}{k+i}^{1/2}
\end{align*}
is precisely the quantity addressed in Lemma \ref{lemma:DoubleProd}. If $\phi_k \phi_{k+i} \neq 0$ and $i\geq 1$ then $i+k\geq K$, hence the assumptions of Lemma \ref{lemma:DoubleProd} hold and we get for all $a\in (a_\phi,1)$: 
\begin{align*}
\Sigma_{a,\phi}(l)&\leq a^{2l}\left(1+\sum_{k=K}^N{\binom{k+l}{k}(1-a^2)^k}\phi_k^2 \right)\\
&+2 \, a^{2l}\left( \underset{0\leq k< k+i \leq N}{\sum \! \sum}\left\{ (1-a^2)^{\frac{2k+i}{2}}|\phi_k| |\phi_{k+i}| C_k C_{k+i} \binom{k+l}{k}^{1/2}\binom{k+l+i}{k+i}^{1/2}\right\}\right).
\end{align*}
Noticing that $C_k \geq 1$ for every positive integer $K$ and that $C_0=1$ allows to recognize the development of a square:
\begin{align*}
\Sigma_{a,\phi}(l)&\leq a^{2l}\left(1+\sum_{k=K}^N{\binom{k+l}{k}^{1/2}(1-a^2)^\frac{k}{2}C_k |\phi_k|} \right)^2\\ &\leq a^{2l}\left(1+\sum_{k=K}^N{\frac{C_k |\phi_k|}{\sqrt{k!}}\left((N+l)(1-a^2)\right)^\frac{k}{2}} \right)^2,
\end{align*}
where we used that for all natural integer $k \leq N$,
\begin{align*}
\binom{k+l}{k} &\leq \frac{(N+l)^k}{k!}.
\end{align*}
We recognize the coefficient $\gamma_k$ introduced in Section \ref{sec:HermiteFourier} to state Hypothesis \textbf{(H)}:
\begin{align*}
\gamma_k&= \frac{C_k |\phi_k|}{\sqrt{k!}},\quad \quad k \geq K,
\end{align*}
so that
\begin{align*}
\Sigma_{a,\phi}(l)&\leq a^{-2N}a^{2(N+l)}\left(1+\sum_{k=K}^N{\gamma_k\left((N+l)(1-a^2)\right)^\frac{k}{2}} \right)^2.
\end{align*}
For all $a\in (0,1)$ and $l \in {\Z^+}$, we perform the change of variables
\begin{align*}
u_{a,\phi}(l):=-(l+N) \log a >0, 
\end{align*}
so that
\begin{align*}
(N+l)(1-a^2)=(N+l)\left(1- \exp \left( -2\frac{u_{a,\phi}(l)}{l+N} \right) \right) \leq 2 u_{a,\phi}(l),
\end{align*}
and introduce the function
\begin{align*}
h(u)=\exp(-u)\left(1+\sum_{k=K}^N{\gamma_k (2u)^{k/2}}\right),\quad \quad u \geq 0.
\end{align*}
Then, 
\begin{align*}
\Sigma_{a,\phi}(l)&\leq a^{-2N} h^2\left( u_{a,\phi}(l)\right).
\end{align*}
The last part of the proof is devoted to showing that the function $h$ is non-increasing on $[0,+\infty)$; indeed in that case, we have for all $a \in (a_\phi,1)$ and $l \geq K$:
\begin{align*}
\Sigma_{a,\phi}(l)& \leq  a^{-2N} h^2\left( u_{a,\phi}(K)\right) =a^{2K} \left(1+\sum_{k=K}^N{\gamma_k \left(-2(K+N)\log a\right)^{k/2}}\right)^2,
\end{align*}
which, jointly with Lemma \ref{lemma:Gershgorin}, proves Proposition \ref{prop:PoincareReverse}. So let us study the variation of $h$. For all $u\geq 0$,
\begin{align*}
h'(u)&=\exp(-u)\left(-1-\sum_{k=K}^N{\gamma_k (2u)^{k/2}}+\sum_{k=K}^N{\gamma_k k(2u)^{(k-2)/2}}\right).
\end{align*}
Let us consider separately the powers of $u^\frac{1}{2}$ ranging from $K$ to $N-2$ (when existing) and the remaining powers:
\begin{align*}
-1-\sum_{k=K}^N{\gamma_k (2u)^{k/2}}+\sum_{k=K}^N{\gamma_k k(2u)^{(k-2)/2}} &=\sum_{k=K}^{N-2}{\left(-\gamma_k+(k+2)\gamma_{k+2}\right) (2u)^{k/2}}\\
&+ K\gamma_K (2u)^{(K-2)/2}+\gamma_{K+1}(K+1)(2u)^{(K-1)/2}\\
&-\left(1+\gamma_N(2u)^{N/2}+\gamma_{N-1}(2u)^{(N-1)/2}\right).
\end{align*}
As Hypothesis \textbf{(H1)} holds, the sum $\sum_{k=K}^{N-2}$ is nonpositive. If $K\leq N-1$, the remaining term writes
\begin{align*}
&\frac{1}{2}\left(-1-2\gamma_N(2u)^{N/2}+2\gamma_{K+1}(K+1)(2u)^{(K-1)/2}\right)\\
+&\frac{1}{2}\left(-1-2\gamma_{N-1}(2u)^{(N-1)/2}+2K\gamma_K (2u)^{(K-2)/2}\right),
\end{align*}
which is nonpositive thanks to Hypothesis \textbf{(H2a)} and Lemma \ref{lemma:Polynom}. If $K=N$, the same arguments provide the nonpositivity of the remaining term, which reduces to
$$-1 -\gamma_N(2u)^{N/2}+N\gamma_N(2u)^{(N-2)/2}.$$
Finally, under Hypothesis \textbf{(H)}, we find that $h$ has a nonpositive derivative on $[0,+\infty[$ hence is non-increasing, which completes the proof.
\end{proof}

\subsection{Proof of Theorem \ref{thm:RefinedConvolChi}}
\label{sec:ProofTheoremConvolChi}

Let us turn to the proof Theorem \ref{thm:RefinedConvolChi}.

\begin{proof}[Proof of Theorem \ref{thm:RefinedConvolChi}] 
Let $\phi,\,f \in L^2(\mu)$ be as in the wording of the theorem. For all $a\in [0,1]$,
\begin{align*}
\qi_2\left (af*\sqrt{1-a^2}\phi\right) &= \left\|K_a(f,\phi) -1\right\|_{L^2(\mu)} 
& \leq \left\|K_a(f-1,\phi) \right\|_{L^2(\mu)}+\left\|K_a(\1,\phi) -1\right\|_{L^2(\mu)}.
\end{align*}
According to Proposition \ref{prop:DefaultInv}, for all $a\in [0,1]$,
\begin{align*}
\left\|K_a(\1,\phi) -1\right\|_{L^2(\mu)}&=  \|Q_{a,\phi}^*[\1]-1\|_{L^2(\mu)} \leq (1-a^2)^\frac{r+1}{2} \qi_2(\phi),
\end{align*}
while by Proposition \ref{prop:PoincareReverse}, for all $a\in (a_\phi,1)$,
\begin{align*}
\left\|K_a(f-1,\phi) \right\|_{L^2(\mu)}&=\|Q_{a,\phi}^*[f-1]\|_{L^2(\mu)}\leq a^{r+1}\left(1+d_\phi(a)\right) \qi_2(f),
\end{align*}
which proves the theorem.
\end{proof}

\subsection{Alternative bound}
\label{sec:AlterBound}

For the sake of completeness, let us conclude this section with a bound alternative to inequality \eqref{eq:ConvolChi}.

\begin{proposition}[Alternative bound on $\chi_2$ under convolution]
Let $f,\phi \in L^2(\mu)$ be density with moments matching the Gaussian moments up to order $r\in {\Z^+}$, and moreover assume that $f$ is $(r+1)$-times derivable, with $D^{r+1}f \in L^2(\mu)$. Then, there exists a universal constant $c_r>0$ such that $\forall a \in (0,1)$,
\begin{align}
\qi_2\left (af*\sqrt{1-a^2}\phi\right) &\leq a^{r+1}\qi_2(f)+(1-a^2)^\frac{r+1}{2}\qi_2(\phi) \nonumber\\
&+c_r (1-a^2)^\frac{r+1}{2}\left( \qi_2(f)\qi_2(\phi) +  	
\left\|D^{r+1}f\right\|_{L^2(\mu)} \left\|D^{-(r+1)}\phi\right\|_{L^2(\mu)}\right),
\label{eq:SymmetricConvolChi}
\end{align}
where $D^{-1}$ stands for the operator which maps a function $\phi \in L^2(\mu)$ onto its primitive with vanishing mean.
\label{prop:AlterBound}
\end{proposition}

Contrarily to what happens for bound \eqref{eq:ConvolChi}, \eqref{eq:SymmetricConvolChi} stands for all $a\in (0,1)$ and the polynomial assumption on $\phi$ is not required, making the relative roles of $f$ and $\phi$ more symmetric. 
The drawback of bound \eqref{eq:SymmetricConvolChi} is that it involves the norm of the $(r+1)$-th derivative of $f$, which we fail to control in the framework of the Central Limit Theorem.

\begin{proof}[Proof of Proposition \ref{prop:AlterBound}]
Set $r\in {\Z^+}$, $K=r+1$ and $f,\phi$ two densities as in the wording of the remark, so that 
\begin{align*}
f=1+\sum_{k=K}^{+\infty}{f_k \H_k},\quad \quad \phi=1+\sum_{k=K}^{+\infty}{\phi_k \H_k}.
\end{align*}
For all $a\in (0,1)$, one has:
\begin{align*}
\qi_2\left(af*\sqrt{1-a^2}\, \phi \right)&=\left\|K_a(f,\phi)-1\right\|_{L^2(\mu)}\\
& \leq \left\|K_a(f-1,\1)\right\|_{L^2(\mu)} +\left\|K_a(\1,\phi-1)\right\|_{L^2(\mu)}+\left\|K_a(f-1,\phi -1)\right\|_{L^2(\mu)}.
\end{align*}
Now, 
\begin{align*}
K_a(f-1,\1)&=P_{-\log a}[f-1],\quad \quad K_a(\1,\phi-1)=P_{-\frac{1}{2}\log (1-a^2)}[\phi-1],
\end{align*}
hence by the improved Poincar\'e inequality from Proposition \ref{prop:DefaultInv} we get the two first terms of the bound. It remains to consider $\left\|K_a(f-1,\phi -1)\right\|_{L^2(\mu)}$. For all $a\in (0,1)$, one has by Remark \ref{rmq:HermiteConvol}:
\begin{align*}
K_a(f-1,\phi-1)&=\sum_{m=K}^{+\infty}\,\sum_{n=K}^{+\infty}\binom{m+n}{m}^{1/2}a^m(1-a^2)^{n/2}f_m\phi_n\H_{n+m},
\end{align*}
hence
\begin{align*}
\left\|K_a(f-1,\phi -1)\right\|_{L^2(\mu)}^2&=\sum_{l=2K}^{+\infty}\left(\sum_{\substack{m+n=l\\m,n\geq K}}\binom{l}{m}^{1/2}a^m(1-a^2)^{n/2}f_m\phi_n\right)^2\\
&=(1-a^2)^{K/2}\sum_{l=2K}^{+\infty}\left(\sum_{\substack{m+n=l-K\\m\geq K,n \geq 0}}\binom{l}{m}^{1/2}a^m(1-a^2)^{n/2}f_m\phi_{n+K}\right)^2.
\end{align*}
By Cauchy-Schwarz inequality and Pascal formula, this rewrites again
\begin{gather*}
(1-a^2)^{K/2}\sum_{l=2K}^{+\infty}\left(\sum_{\substack{m+n=l-K\\m\geq K,n \geq 0}}\binom{l-K}{m}^{1/2}\left(\frac{l\cdots(l-K+1)}{(n+K)\cdots(n+1)}\right)^{1/2}a^m(1-a^2)^{n/2}f_m\phi_{n+K}\right)^2\\
\leq (1-a^2)^{K/2}\sum_{l=2K}^{+\infty}\sum_{\substack{m+n=l-K\\m\geq K,n \geq 0}}\frac{l\cdots(l-K+1)}{(n+K)\cdots(n+1)}f_m^2\phi_{n+K}^2.
\end{gather*}
Now, there exists $c_K>0$ such that $\forall x \geq K,\forall y \geq 0$,
\begin{align*}
(x+y+K)\cdots (x+y+1) &\leq c_K \left(x \cdots (x-K+1)+(y+K)\cdots (y+1) \right).
\end{align*}
Applying this to $x=m$ and $y=n$, we find that
\begin{align*}
\left\|K_a(f-1,\phi -1)\right\|_{L^2(\mu)}^2&\leq c_K(1-a^2)^{K/2}\sum_{l=2K}^{+\infty}\sum_{\substack{m+n=l-K\\m\geq K,n \geq 0}}\left(1+ \frac{m\cdots (m-K+1)}{(n+K)\cdots(n+1)}\right)f_m^2\phi_{n+K}^2\\
&=c_K(1-a^2)^{K/2}  \left(\sum_{m\geq K}{f_m^2}\right)\left(\sum_{n\geq K}{\phi_n^2}\right)\\
& +c_K (1-a^2)^{K/2}  \left(\sum_{m\geq K}{m\cdots (m-K+1) f_m^2}\right)\left(\sum_{n\geq K}{  \frac{\phi_n^2}{(n+K)\cdots(n+1)}}\right),
\end{align*}
which is the Hermite representation of the expected quantity by formula \eqref{eq:DeriveeHermite}.
\end{proof}

\section[Proof of the convergence theorem]{Proof of Theorem \ref{thm:ConvergChi}}
\label{sec:ProofRateChi}
Finally, we conclude the article with the proof of our main theorem, Theorem \ref{thm:ConvergChi}, which follows on from the recursion formula \eqref{eq:RecursionDensity} and barycentric convolution inequality for $\chi_2$-distance \eqref{eq:ConvolChi}.\smallskip

\begin{proof}[Proof of Theorem \ref{thm:ConvergChi}]
In the framework of the theorem, denote 
\begin{align*}
n_0:=\left\lceil \frac{1}{1-a_\phi^2}  \right\rceil \, \vee \, 2.
\end{align*}
By the two aforementioned relations, we have for all integer $n \geq n_0$:
\begin{align*}
\qi_2(f_n) &\leq \left(1-\frac{1}{n}\right)^\frac{r+1}{2}\left(1+d_\phi\left(\sqrt{1-1/n}  \right)\right) \qi_2(f_{n-1})+\frac{1}{n^\frac{r+1}{2}}\qi_2(\phi).
\end{align*}
Remembering that $r\geq 2$, let us call for all $n\geq n_0$,
\begin{align*}
c_n:=\left(1-\frac{1}{n}\right)^\frac{r+1}{2}\left(1+d_\phi\left(\sqrt{1-1/n}  \right)\right)=1-\frac{r+1}{2n}+\mathcal{O}\left(\frac{1}{n^\frac{3}{2}}\right),\quad \quad d_n:=\frac{1}{n^\frac{r+1}{2}}\qi_2(\phi),
\end{align*}
where we denote $v_n=\mathcal{O}(u_n)$ if $\limsup_{n \rightarrow + \infty} |v_n/u_n|<+\infty$. The preceding recursive inequality yields
\begin{align*}
\qi_2(f_n) &\leq \left(\prod_{k=n_0}^n c_k\right) \qi_2(f_{n_0-1}) + \sum_{k=n_0}^n \left(\prod_{j=k+1}^n c_k\right)d_k.
\end{align*}
Now,
\begin{align*}
\log\left( \prod_{k=n_0}^n c_k \right) &=\sum_{k=n_0}^n \log c_k=- \sum_{k=n_0}^n\left(\frac{r+1}{2n}+\mathcal{O}\left(\frac{1}{n^\frac{3}{2}}\right)\right)=-\frac{r+1}{2} \log n + \mathcal{O}\left(1\right).
\end{align*}
This leads to
\begin{align*}
\prod_{k=n_0}^n c_k = \mathcal{O}\left( \frac{1}{n^\frac{r+1}{2}} \right);\quad \quad \sum_{k=n_0}^n \left(\prod_{j=k+1}^n c_k\right)d_k &= \left(\prod_{j=n_0}^n c_j\right) \sum_{k=n_0}^n{\frac{d_k}{\prod_{j=n_0}^k c_j}}\\
&= \mathcal{O}\left( \frac{1}{n^\frac{r+1}{2}} \right) \sum_{k=n_0}^n \mathcal{O}\left(1\right)\\
&=\mathcal{O}\left( \frac{1}{n^\frac{r-1}{2}} \right).
\end{align*}
Finally, 
\begin{align*}
\qi_2(f_n):=\mathcal{O}\left( \frac{1}{n^\frac{r+1}{2}} \right)+\mathcal{O}\left( \frac{1}{n^\frac{r-1}{2}} \right)=\mathcal{O}\left( \frac{1}{n^\frac{r-1}{2}} \right),
\end{align*}
which proves the theorem.
\end{proof}

\begin{remark}[Non i.i.d\@ case]
\label{rmk:NonIIDCase}
If we suppose that the random variables $(X_i)_{i\geq 1}$ are independent and can be distributed according to densities $(\phi_j)_{j \in J}$, where $J$ is a finite set and each $\phi_j$ is polynomial and complies with \textbf{(H)}, then the integer $n_0$ above is replaced by
\begin{align*}
n_0:=\max_{j \in J}\left\lceil \frac{1}{1-a_{\phi_j}^2}  \right\rceil \, \vee \, 2,
\end{align*}
and the quantities $(c_n,d_n)_{n \geq n_0}$ by
\begin{align*}
c_n & :=\left(1-\frac{1}{n}\right)^\frac{r+1}{2}\max_{j \in J} \left(1+d_{\phi_j}\left(\sqrt{1-1/n}  \right)\right)=1-\frac{r+1}{2n}+\mathcal{O}\left(\frac{1}{n^\frac{3}{2}}\right),\\
 d_n & :=\frac{1}{n^\frac{r+1}{2}}\max_{j \in J}\qi_2(\phi_j)= \mathcal{O}\left(\frac{1}{n^\frac{r+1}{2}}\right).
\end{align*}
The rest of the proof is unchanged.
\end{remark}

\bibliographystyle{plain}
\bibliography{BiblioTCL}

\begin{thebibliography}{10}

\bibitem{MicloAlgorithm}
Marc Arnaudon and Laurent Miclo.
\newblock A stochastic algorithm finding {$p$}-means on the circle.
\newblock {\em Bernoulli}, 22(4):2237--2300, 2016.

\bibitem{ArtsteinEntropy}
Shiri Artstein, Keith~M. Ball, Franck Barthe, and Assaf Naor.
\newblock On the rate of convergence in the entropic central limit theorem.
\newblock {\em Probab. Theory Related Fields}, 129(3):381--390, 2004.

\bibitem{BallyVT}
Vlad Bally and Lucia Caramellino.
\newblock Asymptotic development for the {CLT} in total variation distance.
\newblock {\em Bernoulli}, 22(4):2442--2485, 2016.

\bibitem{BallyNonIdenticallyDistributed}
Vlad Bally, Lucia Caramellino, and Guillaume Poly.
\newblock Convergence in distribution norms in the {CLT} for non identical
  distributed random variables.
\newblock {\em arXiv preprint arXiv:1606.01629}, 2017.

\bibitem{BarronEntropy}
Andrew~R. Barron.
\newblock Entropy and the central limit theorem.
\newblock {\em Ann. Probab.}, 14(1):336--342, 1986.

\bibitem{BerryAccuracy}
Andrew~C. Berry.
\newblock The accuracy of the {G}aussian approximation to the sum of
  independent variates.
\newblock {\em Trans. Amer. Math. Soc.}, 49:122--136, 1941.

\bibitem{BobkovRenyi}
S.~G. {Bobkov}, G.~P. {Chistyakov}, and F.~{G{\"o}tze}.
\newblock {R{\'e}nyi divergence and the central limit theorem}.
\newblock {\em ArXiv e-prints}, 2016.

\bibitem{BobkovEntropyRate}
Sergey~G. Bobkov, Gennadiy~P. Chistyakov, and Friedrich G\"otze.
\newblock Rate of convergence and {E}dgeworth-type expansion in the entropic
  central limit theorem.
\newblock {\em Ann. Probab.}, 41(4), 2013.

\bibitem{BobkovFisher}
Sergey~G Bobkov, Gennadiy~P Chistyakov, and Friedrich G{\"o}tze.
\newblock Fisher information and the central limit theorem.
\newblock {\em Probability Theory and Related Fields}, 159(1-2):1--59, 2014.

\bibitem{BrownFisher}
Lawrence~D. Brown.
\newblock A proof of the central limit theorem motivated by the
  {C}ram\'er-{R}ao inequality.
\newblock In {\em Statistics and probability: essays in honor of {C}. {R}.
  {R}ao}, pages 141--148. North-Holland, Amsterdam-New York, 1982.

\bibitem{RioStationarySequence}
J\'er\^ome Dedecker, Florence Merlev\`ede, and Emmanuel Rio.
\newblock Rates of convergence for minimal distances in the central limit
  theorem under projective criteria.
\newblock {\em Electron. J. Probab.}, 14:no. 35, 978--1011, 2009.

\bibitem{EsseenDistribution}
Carl-Gustav Esseen.
\newblock Fourier analysis of distribution functions. {A} mathematical study of
  the {L}aplace-{G}aussian law.
\newblock {\em Acta Math.}, 77:1--125, 1945.

\bibitem{FillReverse}
James~Allen Fill.
\newblock Eigenvalue bounds on convergence to stationarity for nonreversible
  {M}arkov chains, with an application to the exclusion process.
\newblock {\em Ann. Appl. Probab.}, 1(1):62--87, 1991.

\bibitem{Gershgorin}
Semyon Gershgorin.
\newblock \"{U}ber die {A}bgrenzung der {E}igenwerte einer {M}atrix.
\newblock {\em Bulletin de l'Acad\'emie des Sciences de l'URSS}, (6), 1931.

\bibitem{GoudonToscani}
Thierry Goudon, St{\'e}phane Junca, and Giuseppe Toscani.
\newblock Fourier-based distances and {B}erry-{E}sseen like inequalities for
  smooth densities.
\newblock {\em Monatshefte f{\"u}r Mathematik}, 135(2):115--136, 2002.

\bibitem{Ibragimov}
I.~A. Ibragimov.
\newblock On the accuracy of approximation by the normal distribution of
  distribution functions of sums of independent random variables.
\newblock {\em Teor. Verojatnost. i Primenen}, 11:632--655, 1966.

\bibitem{NourdinSteinMalliavin}
Ivan Nourdin and Giovanni Peccati.
\newblock {\em Normal approximations with Malliavin calculus: from Stein's
  method to universality}, volume 192.
\newblock Cambridge University Press, 2012.

\bibitem{NourdinTV}
Ivan Nourdin and Guillaume Poly.
\newblock Convergence in total variation on {W}iener chaos.
\newblock {\em Stochastic Process. Appl.}, 123(2):651--674, 2013.

\bibitem{RioBounds}
Emmanuel Rio.
\newblock Upper bounds for minimal distances in the central limit theorem.
\newblock {\em Ann. Inst. Henri Poincar\'e Probab. Stat.}, 45(3):802--817,
  2009.

\bibitem{RioConstants}
Emmanuel Rio.
\newblock Asymptotic constants for minimal distance in the central limit
  theorem.
\newblock {\em Electron. Commun. Probab.}, 16:96--103, 2011.

\bibitem{ShannonCommunication}
Claude~E. Shannon and Warren Weaver.
\newblock {\em The {M}athematical {T}heory of {C}ommunication}.
\newblock The University of Illinois Press, Urbana, Ill., 1949.

\bibitem{Shirazdinov}
S.~H. Sira\v{z}dinov and M.~Mamatov.
\newblock On mean convergence for densities.
\newblock {\em Teor. Verojatnost. i Primenen.}, 7:433--437, 1962.

\bibitem{StamInequalities}
A.~J. Stam.
\newblock Some inequalities satisfied by the quantities of information of
  {F}isher and {S}hannon.
\newblock {\em Information and Control}, 2, 1959.

\bibitem{Tanaka}
Hiroshi Tanaka.
\newblock An inequality for a functional of probability distributions and its
  application to {K}ac's one-dimensional model of a {M}axwellian gas.
\newblock {\em Z. Wahrscheinlichkeitstheorie und Verw. Gebiete}, 27:47--52,
  1973.

\end{thebibliography}
\end{document}